\documentclass{amsart}[12pt]
\def\doctype{}

\usepackage{latexsym,amssymb,amsmath}
\usepackage{tikz}
\usetikzlibrary{calc}

\usepackage{fancyhdr}
\usepackage{hyperref}

\def\cB{\mathcal{B}}

\def\Z{{\mathbb Z}}
\def\lam{\lambda}

\newcommand{\comment}[1]{}

\allowdisplaybreaks
\numberwithin{equation}{section}

\fancypagestyle{titlepage}{
\fancyhead[R]{\doctype}
\fancyhead[CL]{}
\cfoot{\vspace{5pt} \thepage}
}

\let\oldsection\section
\newcommand\boldsection[1]{\oldsection{\bf #1}}
\newcommand\starsection[1]{\oldsection*{\bf #1}}
\makeatletter
\renewcommand\section{\@ifstar\starsection\boldsection}
\makeatother

\newtheoremstyle{theorem}
  {12pt}		  % space above
  {0pt}  % space below
  {\sl}  % bofy font
  {\parindent}     % ident - empty=no indent,  \parindent= paragraph indent
  {\bf}  % thm head font
  {. }    % punctuation after thm head
  { }    % space after thm head: `` ``=normal \newline=linebreak
  {}     % thm head specification
\theoremstyle{theorem}
\newtheorem{thm}{Theorem}[section]  % 1st argument is your name for it
\newtheorem{lemma}[thm]{Lemma}     % 2nd argument is what is printed

\newtheorem{cons}[thm]{Construction}
\newtheorem{prop}[thm]{Proposition}

\newtheoremstyle{definition}
  {12pt}		  % space above
  {0pt}  % space below
  {}  % bofy font
  {\parindent}     % ident - empty=no indent,  \parindent= paragraph indent
  {\bf}  % thm head font
  {. }    % punctuation after thm head
  { }    % space after thm head: `` ``=normal \newline=linebreak
  {}     % thm head specification
\theoremstyle{definition}

\newcommand\rk{{\sc Remark.} }

\renewcommand{\proofname}{Proof}

\makeatletter
\renewenvironment{proof}[1][\proofname]{\par
  \pushQED{\qed}%
  \normalfont \partopsep=\z@skip \topsep=\z@skip
  \trivlist
  \item[\hskip\labelsep
        \scshape
    #1\@addpunct{.}]\ignorespaces
}{%
  \popQED\endtrivlist\@endpefalse
}
\makeatother

\makeatletter
\renewcommand*\@maketitle{%
  \normalfont\normalsize
  \@adminfootnotes
  \@mkboth{\@nx\shortauthors}{\@nx\shorttitle}%
  \global\topskip42\p@\relax % 5.5pc   "   "   "     "     "
  \@settitle
  \ifx\@empty\authors \else {\vskip 1em
\vtop{\centering\shortauthors\@@par}} \fi
  \ifx\@empty\@date \else {\vskip 1em \vtop{\centering\@date\@@par}}\fi % MY CHANGE
  \ifx\@empty\@dedicatory
  \else
    \baselineskip18\p@
    \vtop{\centering{\footnotesize\itshape\@dedicatory\@@par}%
      \global\dimen@i\prevdepth}\prevdepth\dimen@i
  \fi
  \@setabstract
  \normalsize
  \if@titlepage
    \newpage
  \else
    \dimen@34\p@ \advance\dimen@-\baselineskip
    \vskip\dimen@\relax
  \fi
} % end \@maketitle
\renewcommand*\@adminfootnotes{%
  \let\@makefnmark\relax  \let\@thefnmark\relax
  \ifx\@empty\@subjclass\else \@footnotetext{\@setsubjclass}\fi
  \ifx\@empty\@keywords\else \@footnotetext{\@setkeywords}\fi
  \ifx\@empty\thankses\else \@footnotetext{%
    \def\par{\let\par\@par}\@setthanks}%
  \fi
\thispagestyle{titlepage}
}
\makeatother

\begin{document}

\title[Designs of dimension three]{\large Some new block designs of
dimension three}

\author{Coen del Valle and Peter J.~Dukes}
\address{\rm Department of Mathematics and Statistics,
University of Victoria, Victoria, BC, Canada}
\email{cdelvalle@uvic.ca,dukes@uvic.ca}

\thanks{This research is supported by NSERC grant 312595--2017}

\date{\today}

\begin{abstract}
The dimension of a block design is the maximum
positive integer $d$ such that any $d$ of its points are contained in a
proper subdesign.  Pairwise balanced designs PBD$(v,K)$ have dimension at least two as long as not all points are on the same line.  On the other hand, designs of dimension three appear to be very scarce.  We study designs of dimension three with block sizes in $K=\{3,4\}$ or $\{3,5\}$, obtaining several explicit constructions and one nonexistence result in the latter case.  As applications, we obtain a result on dimension three triple systems having arbitrary index as well as symmetric latin squares which are covered in a similar sense by proper subsquares.
\end{abstract}

\maketitle
\hrule

\section{Introduction}

Let $v$ be a positive integer and $K \subseteq \{2,3,4,\dots\}$.
A \emph{pairwise balanced design} PBD$(v,K)$ is a pair $(X,\cB)$, where $X$ is a
$v$-set of \emph{points} and $\cB$ is a family of {\em blocks} such that
\begin{itemize}
\item
for each $B \in \cB$, we have $B \subseteq X$ with $|B| \in K$; and
\item
any two distinct points in $X$ appear together in exactly one block.
\end{itemize}
These objects are also sometimes known as `linear spaces', where blocks assume the role of lines.

There are arithmetic conditions on $v$ in terms of the set $K$.  If we define 
$\alpha(K):=\gcd\{k-1: k \in K\}$ and $\beta(K):=\gcd\{k(k-1): k \in K\}$, then elementary counting arguments show
\begin{align}
\label{local}
v-1 &\equiv 0 \pmod{\alpha(K)},~\text{and}\\
\label{global}
v(v-1) &\equiv 0 \pmod{\beta(K)}.
\end{align}
Wilson's theory, \cite{RMW2}, establishes that (\ref{local}) and (\ref{global}) are sufficient for large $v$.

In a PBD $(X,\cB)$, a \emph{subdesign} is a pair $(Y,\cB_Y)$, where $Y
\subseteq X$ and $\cB_Y:=\{B \in \cB: B \subseteq Y\}$, which is itself a pairwise balanced design.  Subdesigns are also
called \emph{flats}, especially in the context of linear spaces, and this is the term we mainly use in what follows.

The set of flats in $(X,\cB)$ form a lattice under intersection.  So any set
of points $S \subseteq X$ generates a flat, which we denote by $\langle S \rangle$, equal to the intersection of all flats
containing $S$. We sometimes abuse notation and put a list of elements inside angle brackets.  For $x \in X$, we have $\langle x \rangle = \{x\}$ and for $x, y \in X$ with $x \neq y$, we have that $\langle x,y \rangle$ is the unique block containing $x$ and $y$.  Since our sets here are finite, one can think of $\langle S \rangle$ as the limit of the chain $S=S_0 \subseteq S_1 \subseteq S_2 \subseteq \cdots$, where $S_{i+1} = \cup_{x,y \in S_i} \langle x,y \rangle$ for $i \ge 0$.

The {\em dimension} of a PBD is the maximum integer $d$ such that any set of $d$ points generates a
proper flat.  Any PBD  has dimension at least two, provided that not all points are on the same block.
See \cite{D} and \cite[Chapter 7]{BB} for a discussion of dimension in the context of linear spaces.

The points and lines of either the affine space AG$_d(q)$ or projective space PG$_d(q)$ afford
designs with dimension $d$.  The parameters are special: AG$_d(q)$ leads to a PBD$(q^d,\{q\})$
and PG$_d(q)$ leads to a PBD$(\frac{q^{d+1}-1}{q-1},\{q+1\})$.  More generally, the second author and A.C.H.~Ling showed in 
\cite{DL2} that, given $K$ and $d$, the arithmetic conditions (\ref{local}) and (\ref{global}) are sufficient for existence of a PBD$(v,K)$ of dimension at least $d$ for $v \ge v_0(K,d)$.

A {\em Steiner triple system} is a PBD$(v,\{3\})$.  It is well
known that a Steiner triple system on $v$ points exists if and only if $v
\equiv 1$ or $3 \pmod{6}$.  A {\em Steiner space} is defined to be a Steiner
triple system of dimension at least 3.
Teirlinck in \cite{T} studied the existence of Steiner spaces, finding that,
for $v \not\in \{51,67,69,145\}$, they exist if and only if $v=15,27,31,39$,
or $v \ge 45$.  The four undecided cases are still open, to the best of our
knowledge.

Dukes and Niezen \cite{DN} obtained a nearly complete 
existence theory for the case $K=\{3,4,5\}$ and dimension $3$.
Note that $\alpha(K)=1$ and $\beta(K)=2$ in this case, so (\ref{local}) and
(\ref{global}) disappear and all positive integers are admissible.

\begin{thm}[\cite{DN}]
\label{pbd345a}
There exists a PBD$(v,\{3,4,5\})$ of dimension three if and only if $v=15$
or $v \geq 27$ except for $v=32$ and possibly for $v \in
\{33,34,35,38,41,42,43,47\}$.
\end{thm}

Here, we consider the cases $K=\{3,4\}$ and $K=\{3,5\}$, obtaining two results of a similar style.

\begin{thm}
\label{main}
(a) For $v \equiv 0,1 \pmod{3}$, there exists a PBD$(v,\{3,4\})$ of dimension three if and only if $v=15$
or $v \geq 27$, except possibly for $v \in \{33,34,42,43,54,69,70,72,78\}$.\\
(b) For odd integers $v$, there exists a PBD$(v,\{3,5\})$ of dimension three if and only if $v=15$
or $v \geq 27$, except for $v=33$ and possibly for $v \in \{35,37,41,43,47,51\}$.
\end{thm}

Note that as a result of \cite[Theorem 7.1]{DN}, it is enough to obtain constructions for dimension at least $3$, since that result facilitates the reduction of dimension.

The outline of the paper is as follows.  Section 2 reviews some background useful for our constructions.  Section 3 gives constructions of various designs for Theorem~\ref{main} not already covered by earlier work.  In Section 4, we provide a nonexistence result for a dimension three PBD$(33,\{3,5\})$.  As an application of Theorem~\ref{main}, we obtain a nearly complete existence theory for dimension three triple systems of general index $\lam$, i.e. 3-uniform set systems in which any two distinct points are together in exactly $\lam$ blocks.  Another application is the construction of symmetric latin squares which are covered by proper subsquares.  These are given in Section 5.

\section{Background}

We begin by stating the basic existence result for PBDs having block sizes in 
$K=\{3,4,5\},\{3,4\}$, or $\{3,5\}$. Proofs and more information can be found in \cite{handbook}.

\begin{thm}
\label{pbd345}
(a) There exists a PBD$(v,\{3,4,5\})$ if and only if $v \neq 2,6,8$.\\
(b) For $v\equiv 0,1\pmod{3}$ there exists a PBD$(v,\{3,4\})$ if and only if $v\ne 6$.\\
(c) For all odd $v$ there exists a PBD$(v,\{3,5\})$.
\end{thm}

%Unless otherwise mentioned, $K=\{3,4,5\}$.
For what follows, we recall the structure of PBD$(v,K)$ with $K \subseteq \{3,4,5\}$ and small $v$.
The unique such PBDs for $v=7$ and $9$ are the `Fano plane' PG$_2(2)$ and the
affine plane AG$_2(3)$, respectively; in these cases all blocks have size three.  
The unique PBD with $v=10$ arises from the extension of one
parallel class in AG$_2(3)$.  The extended blocks have size four.
The only PBD, up to isomorphism, for $v=11$ has one block of size five and all other blocks of
size three. 

A \emph{group divisible design} (GDD) is a triple $(X,\Pi,\cB)$, where $X$ is a
set of \emph{points}, $\Pi$ is a partition of $X$ into \emph{groups}, and
$\cB$ is a set of \emph{blocks} such that
\vspace{-11pt}
\begin{itemize}
\item
a group and a block intersect in at most one point; and
\item
any two points from distinct groups appear together in exactly one block.
\end{itemize}
\vspace{-11pt}
To specify a set $K$ of allowed block sizes, we use the notation $K$-GDD.
The \emph{type} of a GDD is the list of its group sizes.  When this list
contains, say, $u$ copies of the integer $g$, this is abbreviated with
`exponential notation' as $g^u$.  

We remark that a $K$-GDD of type $1^v$ is just a PBD$(v,K)$.
More generally, if a PBD$(v,K)$ has a partition into subdesigns, which may be singletons or blocks, 
then they can be removed and turned into groups to produce a $K$-GDD on $v$ points.
Or, a GDD can be constructed from a PBD by deleting a point $x$ and all its incident
blocks.

%A \emph{transversal design} TD$(k,n)$ is a $\{k\}$-GDD of type $n^k$.  In
%this case, every block meets every group in one point.  For our purposes, it
%is convenient to state the known existence theory for block size $5$.

%\begin{thm}
%\label{td5}
%There exists TD$(5,n)$ if $n \neq 2,3,6,10$.
%\end{thm}

%\rk  Except for the case $n=10$, which still remains in doubt, the converse
%also holds.

We now review some standard design-theoretic constructions.
First, we can `fill' the groups of a GDD with PBDs.
\begin{cons}[Filling groups]
\label{fill-gps}
Suppose there exists a $K$-GDD on $v$ points with group sizes in $G$.  If,
for each $g \in G$:
\vspace{-11pt}
\begin{itemize}
\item
there exists a PBD$(g,K)$, then there exists a PBD$(v,K)$;
\item
there exists a PBD$(g+1,K)$, then there exists a PBD$(v+1,K)$; and
\item
there exists a PBD$(g+h,K)$ containing a flat of order $h$, then there
exists a PBD$(v+h,K)$.
\end{itemize}
\end{cons}

From a PBD or GDD, one may truncate a subset $A \subseteq
X$, replacing blocks $B \in \cB$ by new blocks $B \setminus A$ (and ignoring blocks of size 0 or 1.)  
For our applications of truncation, we want to avoid blocks of size two.

The next construction builds larger GDDs from smaller ones.

\begin{cons}[Wilson's fundamental construction]
\label{wfc}
Suppose there exists a `master' GDD $(X,\Pi,\cB)$, where
$\Pi=\{X_1,\dots,X_u\}$.  Let $\omega:X \rightarrow \{0,1,2,\dots\}$,
assigning nonnegative weights to each point in such a way that for every $B
\in \cB$ there exists an `ingredient' $K$-GDD of type $\omega(B):=[\omega(x)
\mid x \in B]$.  Then there exists a $K$-GDD of type $$\left[\sum_{x \in
X_1} \omega(x),\dots,\sum_{x \in X_u} \omega(x)\right].$$
\end{cons}

\rk
Truncation can be viewed as a special case of Construction~\ref{wfc}, where
weights in $\{0,1\}$ are used and the ingredients are simply blocks or shortened blocks.

Later, we use various small GDDs as ingredients in the above construction.  The following can be easily found by truncating points from known designs and/or turning disjoint blocks into groups.

\begin{lemma}[see also \cite{Niezen,T}]
\label{ingredients}
The following group divisible designs exist:\\
(a) 3-GDDs of types $2^3,3^3,3^5,4^3,2^4,4^4,2^14^3,2^34^1$;\\
(b) $\{3,4\}$-GDDs of types $ 3^4,2^3 3^1, 2^1 3^3,3^3 4^1,3^1 4^3, 2^4 3^1$;\\
(c) $\{3,5\}$-GDDs of types $2^5$ and $1^i 3^{5-i}$ for each $i \in \{0,1,\dots,5\}$.
\end{lemma}

Let us call a weighting of points $\omega: X
\rightarrow \Z_{\ge 0}$ \emph{nondegenerate} if the points of nonzero weight are contained in no proper flat.

\begin{prop}[\cite{DL2,DN}]
\label{wfc-dim}
Suppose a nondegenerate weighting is applied to a PBD $(X,\cB)$ of dimension
$d$.  The result of Wilson's fundamental construction is a GDD of dimension
at least $d$.  Moreover, any set of $d$ points are contained in a proper
sub-GDD that intersects each group in zero or all points.
\end{prop}

Observe that if we have a GDD satisfying the conclusion of Proposition~\ref{wfc-dim}, then filling groups as in
Construction~\ref{fill-gps} results in a PBD of dimension $d$.  A typical construction sequence begins with a PBD, applies Construction~\ref{wfc} with some weights to produce a GDD, and finishes with Construction~\ref{fill-gps} to produce a new PBD.  If the input PBD has dimension three and the weighting is nondegenerate, then the resultant PBD has dimension three.

\section{New constructions}

Since Steiner spaces qualify for both of the sets of block sizes we are considering, we need not construct designs with $v\equiv 1,3 \pmod{6}$, where $v=15,27,31,39$, or $v \ge 45$, unless $v$ is in the set of possible exceptions $\{51,67,69,145\}$. In this section we consider most of the remaining values of $v$, relying on truncations and weightings of the affine and projective spaces, appealing to Construction~\ref{wfc}.

\subsection{Block sizes in $K=\{3,4\}$}

We divide the unsettled values for Theorem~\ref{main}(a) into several propositions, organized by construction method.  Recall that $v
\equiv 0,1\pmod{3}$ in this case.

First, we obtain all but finitely many values of $v$ directly from the case $K=\{3,4,5\}$.

\begin{prop}[see also \cite{DN}]
\label{finite-34pbd}
There exists a PBD$(v, \{3,4\})$ of dimension three for $v\in \{45,46,81,$
$82,84,85,87,88,90,91,93,94,108,109,111,112,117,118,120,121,132,133,135,136,138,139\}$ and all $v \ge 144$.
\end{prop}

\begin{proof}
If we give weight 3 to a PBD$(u,\{3,4,5\})$ of dimension three, add zero or one point, and replace blocks with $\{3,4\}$-GDDs of type $3^3$, $3^4$, $3^5$, the result is, after filling groups, a PBD$(v,\{3,4\})$ with $v=3u$ or $3u+1$.  The result then follows from Theorem~\ref{pbd345a}.
\end{proof}

To realize the largest remaining values, we weight truncations of the projective space PG$_3(4)$.

\begin{prop}
There exists a PBD$(v, \{3,4\})$ of dimension three for all $v\in \{114,126,130,142\}$.
\end{prop}

\begin{proof}
We work from PG$_3(4)$, a PBD$(85,\{5\})$ of dimension three.
To get $v=114$ and $126$, truncate all but 3 collinear points or a Fano subplane from one plane PG$_2(4)$, and all but 4 collinear points of another plane. Give weight 3 to the points remaining in the first plane, and weight 2 to all other points. Apply Construction~\ref{wfc}, using $\{3,4\}$-GDDs of types $2^3,3^3,2^4,2^33^1,2^43^1$ as ingredients (see (a) and (b) of Lemma~\ref{ingredients}).  Add one point and fill groups with blocks incident at this point as in Construction~\ref{fill-gps}.
For the other values we start with AG$_3(4)$, a PBD$(64,\{4\})$ of dimension three (and truncation of the projective space above). For $v=130$, give all points weight 2 except one which we give weight 3 and use ingredient GDDs as above.  For 142, give all but 3 collinear points of some plane weight 3, and all remaining points weight 2.  This time we also use an ingredient GDD of type $2^13^3$. Again, add a point and fill the groups.
\end{proof}

We also obtain a PBD$(124,\{3,4\})$ of dimension three by giving weight 4 to a Steiner space of order 31.

\begin{prop}
There exists a PBD$(124, \{3,4\})$ of dimension three.
\end{prop}

Four more orders can be constructed from weightings of PG$_3(3)$.

\begin{prop}
There exists a PBD$(v, \{3,4\})$ of dimension three for all $v\in \{96,100,102,106\}$.
\end{prop}

\begin{proof}
Start with PG$_3(3)$, a PBD$(40,\{4\})$ of dimension three.  The points can be partitioned into a copy of AG$_3(3)$ and a plane PG$_2(3)$. Give each of the 27 points of the affine space weight 3. In the plane, give 1, 3, or 4 collinear points weight 0, and all other points weight 2. Apply Construction~\ref{wfc} using $\{3,4\}$-GDDs of types $2^3,3^3,2^4,2^13^3$ as ingredients.  To the resulting GDDs, add a point and fill groups with blocks of size $3$ or $4$. This gives the largest three values. For $v=96$, truncate from PG$_3(3)$ all points of a plane except those on one line $\ell$. Give all points weight 3, except three of the points on $\ell$, which are given weight 4, using $\{3,4\}$-GDDs of types $3^3,3^4,3^34^1,3^14^3$ as ingredients (see Lemma~\ref{ingredients}). Finally, turn groups of the resulting GDD into blocks.
\end{proof}

Several smaller values come from truncations of projective and affine spaces of dimension three. 

\begin{prop}[\cite{DN}]
There exists a PBD$(v,\{3,4\})$ of dimension three for $v \in \{28,30,31,36,37,$ $39,40,60,61,63,64\}$.
\end{prop}

There are now only seven values left to consider before completing the proof of Theorem~\ref{main}(a). We get $v\in \{48, 51, 52\}$ as in Niezen's thesis \cite{Niezen}, by truncating planes from PG$_3(4)$. The remaining four values are considered in the following proposition.

\begin{prop}
There exists a PBD$(v,\{3,4\})$ of dimension three for $v\in\{58,66,67,76\}$.
\end{prop}

\begin{proof}
We first consider $v=76$.  Start with PG$_3(3)$ and truncate 3 collinear points.  Give the remaining point from that block weight 3 and all other points weight 2. Apply Construction~\ref{wfc} using $\{3,4\}$-GDDs of types $2^3,2^4,2^33^1$ as ingredients. To finish the construction, add a point and fill the groups with blocks of size $3$ or $4$.  For the smaller three values, truncate all but 1 or 4 collinear points from a plane, giving all points weight 2 except for the single point, or 3 or 4 of the 4 points, which will get weight 3. Apply Wilson's Fundamental Construction again, using the same ingredients as before, as well as possibly type $2^13^3$.  The resulting GDDs have $57,65$ or $66$ points and groups sizes in $\{2,3\}$, which are to be filled as before.
\end{proof}

\subsection{Block sizes in $K=\{3,5\}$} We now move on to Theorem~\ref{main}(b).  In what follows we consider only odd integers $v$.
In fact, for $v\geq45$, $v\not\in\{51,67,69,145\}$, we actually only need to consider the congruence class $v\equiv 5\pmod{6}$, as Steiner spaces cover the other possibilities. 

As before, we split the proof into a few separate constructions. First, all but finitely many cases can be handled as in Proposition~\ref{finite-34pbd}, where weight 2 and ingredient $\{3,5\}$-GDDs of type $2^3$, $2^4$ and $2^5$ are used. 

\begin{prop}[see also \cite{DN}]
\label{finite-35pbd}
There exists a PBD$(v,\{3,5\})$ of dimension three for $v=29, 59$, and all odd $v \ge 97$.
\end{prop}

A range of orders can be constructed by weighting PG$_3(3)$ and PG$_3(4)$.

\begin{prop}
There exists a PBD$(v,\{3,5\})$ of dimension three for all $v\in \{65,67,69,71,77,83\}$.
\end{prop}

\begin{proof}
Start with PG$_3(3)$, a PBD$(40,\{4\})$ of dimension three. To get the larger two values, truncate 0 or 3 collinear points giving exactly one point weight 4 (the remaining point from the truncated line in the latter case), and all other points weight 2. Apply Construction~\ref{wfc}, using 3-GDDs of types $2^3,2^4,2^34^1$ from Lemma~\ref{ingredients} as ingredients. Add a point and fill groups with blocks at this new point. For the smaller values, truncate all but three or four collinear points of a plane. If three, give these points weight 4 and all others weight 2. If four, give one, three, or all of these four points weight 4, and the rest weight 2. Apply Wilson's Fundamental Construction again, using the same ingredient GDDs, as well as possibly 3-GDDs of types $4^3,4^4$ and $2^14^3$. Add a point and fill groups with blocks.
\end{proof}

\begin{prop}
There exists a PBD$(v,\{3,5\})$ of dimension three for $v=89$ and $95$.
\end{prop}

\begin{proof}
Start with PG$_3(4)$, a PBD$(85,\{5\})$ of dimension three.  Give $i$ points weight 3, and the rest of the points weight 1 for $i\in \{2,5\}$.  Apply Construction~\ref{wfc} with ingredients from Lemma~\ref{ingredients}(c). Turning the groups of size three into blocks completes the proof.
\end{proof}

The only remaining value to consider here is $v=53$, which is handled by a truncation.

\begin{prop}
There exists a PBD$(53,\{3,5\})$ of dimension three.
\end{prop}

\begin{proof}
Truncate two planes from PG$_3(4)$, leaving behind only their common block. This leaves blocks of size 3 across the remaining three planes, and blocks of size 5 within the remaining planes.  The result is a PBD$(v,\{3,5\})$ with $v=85-2\times 16 = 53$ points.
\end{proof}

This completes our constructions of PBDs of dimension three with block sizes in $\{3,4\}$ and $\{3,5\}$.  Theorem~\ref{main} is proved except for nonexistence in the case $v=33$, $K=\{3,5\}$, which we address in the next section.

\section{Nonexistence}

The second author and Niezen showed in ~\cite{DN} the nonexistence of PBDs of dimension three with block sizes in $\{3,4,5\}$ for $v=32$, and all $v<27,v\ne15$. Since both sets of block sizes considered in this paper are subsets of $\{3,4,5\}$, the nonexistence results from ~\cite{DN} carry over to our work. In this section we add to those results the nonexistence of a PBD$(33,\{3,5\})$ of dimension three. The proof relies on a well-known upper bound for the size of proper flats in a design. In what follows, we say that a block `touches' a flat when it intersects that flat in a single point.  A proof of the following can be found in \cite{DN}.

\begin{lemma}
\label{subsystem-bound}
In a PBD$(v,\Z_{\ge 3})$ with a proper flat $W$ of size $w$, we have $v \ge
2w+1$, with equality if and only if every block intersects $W$, and all
blocks which touch $W$ have size exactly three.
\end{lemma}

Consider any block $B$ in a PBD, say $(X, \mathcal{B})$, of dimension
at least three. For each $x \in X \setminus B$, the flat $F_x = \langle B, x\rangle$ is proper in $X$, and hence these flats partition the points in $X\setminus B$. By deleting $B$ and these flats we get a GDD whose groups partition is $\{F_x \setminus B: x \in X\}$.  This observation is used repeatedly in our structural arguments to follow.

\begin{prop}
\label{nonex-pbd33}
There does not exist a PBD$(33,\{3,5\})$ of dimension three.
\end{prop}

\begin{proof}
Suppose, for contradiction, that such a design exists. Teirlinck showed in ~\cite{T} the nonexistence of a Steiner space of order 33, so our design necessarily has a block of size 5, say $B$. The smallest $\{3,5\}$-PBD containing a block of size 5 has order 11.  So, it follows from this and Lemma~\ref{subsystem-bound} that $B$ can only be in flats of sizes $15,13$, or $11$. This leaves us with only two possible GDD types upon deleting $B$ and its incident proper flats, namely $6^28^2$ and $6^310^1$. To eliminate the first case notice flats of size 11 contain exactly one block of size 5, and a flat of size 13 containing a block of size 5 contains exactly three. This gives that the design arising from the GDD of type $6^28^2$ has exactly five blocks of size 5, which is not possible since we need the number of blocks of size 5 to be a multiple of 3. The second case requires much deeper analysis.

Consider a flat $Y$ of size 15 containing $B$, and hence containing at least three blocks of size $5$.  The block $B$ induces a partition $\{A_1,A_2,A_3\}$ of $X \setminus Y$ into three $6$-subsets according to the flats of order 11 containing $B$.  Consider a different block $B'$ of size $5$ in $Y$.  It induces another such partition $\{A'_1,A'_2,A'_3\}$.  The unique structure of PBD$(11,\{3,5\})$ forces every block of size three in these flats to have exactly one point of $B$ and two points in $X\setminus Y$.  So for any $i$ and $j$, we have $|A_i \cap A'_j| \le 2$, or else the flats $A_i \cup B$ and $A'_j \cup B'$ would overlap in too many points.  It follows from counting that $|A_i \cap A'_j|=2$, and the 18 points of $X \setminus Y$ fall into three `rows' and three `columns' containing six points each, as shown in Figure~\ref{33-structure}.  Without loss of generality, a third block $B''$ of size $5$ can only induce a partition running across these as three `diagonals'.  So $B,B'$ and $B''$ intersect in a common point, say $\infty$.  And we see that more than three (hence at least six for divisibility) blocks of size $5$ in $Y$ is impossible, since there is not enough room for the flats of size 11 from these blocks.

\begin{figure}[htbp]
\begin{center}
\begin{tikzpicture}
\foreach \a in {-1,0,1}
 \foreach \b in {-1,0,1}
  \foreach \c in {-.2,.2}
   \filldraw (\a+\c,\b+\c) circle [radius=.05];

\draw[rounded corners] (-1.4,-1.5) rectangle (-.6,1.5);
\draw[rounded corners] (-.4,-1.5) rectangle (.4,1.5);
\draw[rounded corners] (.6,-1.5) rectangle (1.4,1.5);
\draw[rounded corners] (-1.5,-1.4) rectangle (1.5,-.6);
\draw[rounded corners] (-1.5,-.4) rectangle (1.5,.4);
\draw[rounded corners] (-1.5,.6) rectangle (1.5,1.4);

\foreach \a in {1,2,3}
 \node at (-1.9,2-\a) {$A_{\a}$};
\foreach \a in {1,2,3}
 \node at (\a-2,1.8) {$A'_{\a}$};

\draw[rounded corners] (2,-2.3) rectangle (4,2);

\node at (4.3,-1.8) {$Y$};

\node at (3,1.75) {$\infty$};
\filldraw (3,1.5) circle [radius=.05];
\filldraw (3.75,.75) circle [radius=.05];
\filldraw (3.75,0) circle [radius=.05];

\foreach \a in {2.25,2.75,3.25}
 \foreach \b in {-1.5,-.75,0,.75}
  \filldraw (\a,\b) circle [radius=.05];

\node at (2.25,-2) {$B$};
\node at (2.75,-2) {$B'$};
\node at (3.25,-2) {$B''$};
\node at (3.75,-.5) {$B_3$};

\draw (2.25,-1.5)--(2.25,.75);
\draw (2.75,-1.5)--(2.75,.75);
\draw (3.25,-1.5)--(3.25,.75);
\draw (3.75,0)--(3.75,.75);
\draw (2.25,.75) to [out=90,in=210] (3,1.5);
\draw (2.75,.75) to [out=90,in=250] (3,1.5);
\draw (3.25,.75) to [out=90,in=290] (3,1.5);
\draw (3.75,.75) to [out=90,in=330] (3,1.5);
\end{tikzpicture}
\end{center}
\caption{structure in a hypothetical PBD$(33,\{3,5\})$ of dimension three}
\label{33-structure}
\end{figure}
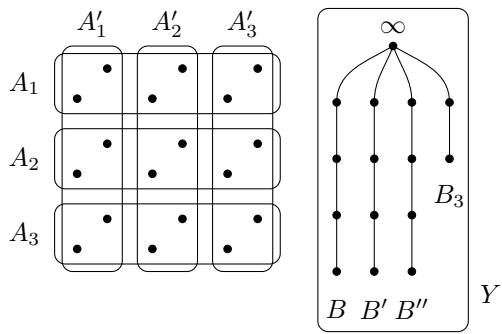

We know, then, that the flat $Y$ has three blocks $B,B',B''$ of size 5 and one block of size 3, say, $B_3$, all containing $\infty$. Notice that $\infty$ induces a partition on the 18 points outside $Y$ into nine pairs.  Any flat containing one of these groups must also then contain the point $\infty$. Also, each point in a block of size 5 contributes 9 pairs between the 18 external points, giving a total of $9\times 13=117$ external pairs covered, leaving 36 remaining.

Consider $B_3$ and its incident proper flats. There are three possible combinations of flats (other than $Y$) incident with $B_3$: three flats of size 9, three of size 7 and one of size 9, or one of size 13 and two of size 7.
We show that only the first possibility can occur.
 %In each of the flats of size 7, exactly four new external pairs are covered, and in a flat of size 9, twelve new external pairs are covered.  This eliminates the second case. 
Since the three blocks of size five cover the `rows', `columns', and one class of `diagonals', we're left to cover 36 pairs along the `back diagonals' with $B_3$. It follows that in a flat of size 7 with $B_3$ the other four points arrive as two of the pairs induced by $\infty$, aligning on the back diagonals. Considering three flats of size 7, we are left with three of the pairs, one from each back diagonal line, together in the flat of size 9 with $B_3$, such a structure necessarily covers pairs already covered with blocks of size five.  This eliminates the second case.  For the third case, a flat of size 13 would cover $\binom{10}{2}-10/2 = 40>36$ new pairs in $X \setminus Y$, and hence this case is not possible either.  

It remains to show that when $B_3$ belongs to three flats of size 9 that such a design fails to have dimension three.
Take any point $x\in B_3$, $x\ne\infty$. Consider the flat generated by $x$, a point $y\ne\infty$ on a block of size 5, and a point $z$ outside $Y$. The size of this flat is an odd integer at least 7. Since we can partition $X \setminus Y$ based on which flat a point generates with $x$ and $y$, we have the cases of a flat of size 13 and two flats of size 7, three of size 7 and one of size 9, or three of size 9. In the first case, a flat of size 13 can intersect $Y$ in exactly one block $\langle x,y \rangle$, leaving 10 external points.  By the pigeonhole principle, some two of them are in a triple with $\infty$, and thus $\langle x,y,z\rangle$ is the whole space. For the other two cases, suppose $\langle x,y,z\rangle$ has size 9. This flat, being a copy of AG$_2(3)$, contains a line $\ell$ of three points in $X \setminus Y$.  Therefore, $\ell$ is contained in an $AG_2(3)$ with $B_3$, hence $\infty$ is again generated and our subsystem $\langle x,y,z\rangle$ is not proper.  This completes the proof.
\end{proof}

%This completes the proof of our main result, Theorem 1.2. 

We are unable to exclude the possibility of a PBD$(35,\{3,5\})$ of dimension three, but if one exists then deleting a block $B$ of size 5 induces a $\{3,5\}$-GDD of one of the following types: $6^5, 8^3 6^1, 10^1 8^1 6^2, 12^1 6^3$ or $10^3$. This list can be reduced slightly with deeper analysis.

We can quickly discard the $10^1 8^1 6^2$ case, as such a design would have five or eight blocks of size 5, contradicting that the number of such blocks must be $1 \pmod{3}$.  Additionally, using structural arguments on flat intersections it is possible to show flats of size 15 can only occur having three pairwise disjoint blocks of size 5.  Consider now the $8^3 6^1$ case.  The flats of size 13 require more blocks of size 5.  These, it can be shown, must intersect in a single point, say $\infty$, on $B$.  Let $x$ be another point on a different block of size 5, let $y$ be in the group of size 6, and consider the flat $Y=\langle \infty,x,y \rangle$.  For $Y$ to be proper, we must have $11 \le |Y| \le 17$ and also $Y \cap B = \{\infty\}$.  But then there is no way to allocate the points of $Y$ to flats induced by the GDD without more than one block occurring in one of these flats.  This eliminates the $8^3 6^1$ case.

For the still open possible types $10^3$, $6^5$ or $12^16^3$, perhaps some new ideas could either eliminate these three cases or lead to a construction.

%This concludes our nonexistence results, we move on to applications of our designs.

\section{Applications}

\subsection{Triple systems of general index}
\label{triple-systems}

A balanced incomplete block design BIBD$(v,k,\lambda)$, or simply $(v,k,\lambda)$-\emph{design} is a pair $(X,\cB)$, where $X$ is a
$v$-set of \emph{points} and $\cB$ is a family of {\em blocks} where
\begin{itemize}
\item
for each $B \in \cB$, we have $B \subseteq X$ with $|B|=k$; and
\item
any two distinct points in $X$ appear together in exactly $\lambda$ blocks.
\end{itemize}
The parameter $\lam$ is often called the \emph{index} of the design.
As with pairwise balanced designs, there are divisibility conditions on the parameters of BIBDs.  The existence of a $(v,k,\lambda)$-design implies
\begin{align}
\label{bibd-local}
\lambda(v-1) &\equiv 0\pmod{k-1}, ~\text{and}\\
\label{bibd-global}
\lambda v(v-1)&\equiv 0\pmod{k(k-1)}.
\end{align}
In the case $k=3$, conditions \eqref{bibd-local} and \eqref{bibd-global} are equivalent to $\lam_{\text{min}} \mid \lam$, where
$$\lam_{\text{min}} = \begin{cases}
1 & \text{if~} v \equiv 1~\text{or}~3 \pmod{6},\\
2 & \text{if~} v \equiv 0~\text{or}~4 \pmod{6},\\
3 & \text{if~} v \equiv 5 \pmod{6},\\
6 & \text{if~} v \equiv 2 \pmod{6}.
\end{cases}$$
%Steiner spaces are just $(v,3,1)$-designs of dimension three. 
In this section we generalize Steiner spaces by considering the cases $\lambda\in\{2,3,6\}$, which by the above remarks essentially settles all values of $\lam$ via taking `copies' of the block family of a $(v,3,\lam_{\text{min}})$ design. Before considering these designs further, we should clarify our definition of dimension. A subdesign in a $(v,k,\lambda)$-design $(X,\cB)$ is a pair $(Y,\cB_Y)$, where $Y \subseteq X$ and $\cB_Y:=\{B \in \cB: B \subseteq Y\}$, which is $(w,k,\lambda)$-design with $w\leq v$. With `subdesign' replacing `flat', the notion of dimension extends naturally to the setting of general $\lam$; that is, the dimension equals the maximum integer $d$ such that any set of $d$ points are contained in a proper subdesign (with the same index).

In the case $k=3$ and $\lambda=2$, \eqref{bibd-local} and \eqref{bibd-global}  become $v\equiv 0,1 \pmod{3}$, the same necessary conditions for a PBD$(v,\{3,4\})$. Now, suppose we take a PBD$(v,\{3,4\})$ of dimension three, and repeat each block of size 3 twice, and replace blocks of size 4 with $(4,3,2)$-designs. The result is then a $(v,3,2)$-design of dimension three.  Hence, our existence results for PBD$(v,\{3,4\})$ of dimension three imply the existence of $(v,3,2)$-designs of dimension three.

Likewise, for $\lambda=3$ one has the same necessary conditions as for PBD$(v,\{3,5\})$, namely that $v$ be odd.  A construction follows from PBD$(v,\{3,5\})$ similar to the above, where we repeat each block of size 3 three times, and replace blocks of size 5 with $(5,3,3)$-designs.  For $\lam=6$, all positive integers are eligible; we use a PBD$(v,\{3,4,5\})$ this time, replacing each block of size $k$ with a $(k,3,6)$-design.
%repeating each block of size 3 six times, fill blocks of size 4 with three $(4,3,2)$-designs, and fill blocks of size 5 with two $(5,3,3)$-designs. 
This, along with Theorem~\ref{pbd345a} and Theorem~\ref{main} gives us the following existence results.

\begin{prop}
(a) For $v \equiv 0,1 \pmod{3}$, there exists a $(v,3,2)$-design of dimension three for $v=15$
and all $v \geq 27$ except possibly for $v \in \{33,34,42,43,54,69,70,72,78\}$.\\
(b) For odd integers $v$, there exists a $(v,3,3)$-design of dimension three for $v=15$
and all $v \geq 27$ except possibly for $v \in \{33,35,37,41,43,47,51\}$.\\
(c) There exists a $(v,3,6)$-design of dimension three for $v=15$
and all $v \geq 27$ except possibly for $v \in\{32,33,34,35,38,41,42,43,47\}$.
\end{prop}

Since $v$ determines the allowed values of $\lam$, and since we may take copies while preserving dimension three, we obtain existence of $(v,3,\lam)$-designs of dimension three whenever $\lam_{\rm min} \mid \lam$, unless $v$ is in the union of the lists of possible exceptions above.
Nonexistence in some small cases follows from nonexistence of the corresponding PBDs.

\begin{prop}
\label{nonex-small}
There does not exist a $(v,3,\lam)$-design of dimension three for $v<27,v\ne 15$.
\end{prop}

\begin{proof}
Consider a hypothetical $(v,3,\lam)$-design of dimension three with $v<27,v\ne15$. By Lemma~\ref{subsystem-bound} (which holds for general $\lambda$ by an analogous proof), any proper subdesign has size at most $\frac{25-1}{2}=12$. In the same way, since our design has dimension three, any subdesign generated by two points has at most $\lfloor\frac{12-1}{2}\rfloor=5$ points.  
Now, any two of these two-point-generated subdesigns which are not equal intersect in at most one point.  So, if we replace every such subdesign with a single block, the result is a PBD$(v,\{3,4,5\})$ of dimension three.  This contradicts Theorem~\ref{pbd345a}.
\end{proof}

%Similar proofs give us the same results for $\lambda\in \{3,6\}$.  
We can additionally eliminate two more cases with some extra analysis.

\begin{prop}
\label{nonex-ts33}
There does not exist a $(33,3,\lam)$-design of dimension three with odd $\lam$.
\end{prop}

\begin{proof}
Consider a hypothetical such design $(X,\mathcal{B})$. Arguing as in the proof of Proposition~\ref{nonex-small}, two-point-generated subdesigns have size at most $7$.  Since $\lam$ is odd, \eqref{bibd-local} leaves only the possibilities $3,5,7$.  Suppose two points $x$ and $y$ generate a subdesign $(Y,\mathcal{B}_Y)$ of size 7. The points of $X\setminus Y$ fall into distinct subdesigns of the form $\langle Y,z\rangle$, $z\in X\setminus Y$, all of which necessarily have size 15.  But this is impossible as $33-7=26$ is not a multiple of $15-7=8$.  It follows that there is no subdesign of size 7 generated by two points. We now proceed as in the proof of Proposition~\ref{nonex-small}, getting a PBD$(33,\{3,5\})$ of dimension three and contradicting Proposition~\ref{nonex-pbd33}.
\end{proof}

\begin{prop}
There does not exist a $(32,3,\lam)$-design of dimension three.
\end{prop}

\begin{proof}
Again, suppose $(X,\mathcal{B})$ is such a design. Possible sizes for two-point-generated subdesigns are $3,4,5,6,7$. The possibility of a subdesign of size 7 is eliminated by a similar analysis as in the proof of Proposition~\ref{nonex-ts33}.  
Suppose two points $x$ and $y$ generate a subdesign $(Y,\mathcal{B}_Y)$ of size 6. The points of $X\setminus Y$ are partitioned by the subdesigns $\langle Y,z\rangle$, $z\in X\setminus Y$, which have sizes in $\{13,14,15\}$. We therefore require a solution to $26=7a+8b+9c$, the only one in nonnegative integers being $a=0,b=1,c=2$.  This is not feasible by a counting argument. Therefore there is no subdesign of size 6 generated by two points. 
Now we again proceed as in the proof of Proposition~\ref{nonex-small}, getting a PBD$(32,\{3,4,5\})$ of dimension three and contradicting Theorem~\ref{pbd345a}.
\end{proof}

We summarize the results of this section in Table~\ref{ts-table}, for which we assume $\lam_{\rm min} \mid \lam$.

\begin{table}[htbp]
\begin{tabular}{l|c|c}
number of points & existence & nonexistence\\
\hline
$v \in \{\dots,14,16,\dots,26,32\}$ & --- & any $\lam$	\\
$v =33$ & ? & odd $\lam$ \\
$v =51$  & $2 \mid \lam$ & ? \\
$v =69$  & $3 \mid \lam$ & ? \\
$v \in \{34,35,37,38,41,43,47,70,72,78\}$ & ? & ? \\
otherwise & any $\lam$ & ---	\\
\end{tabular}
\caption{status of $(v,3,\lam)$-designs of dimension three}
\label{ts-table}
\end{table}
Triple systems of dimension three  which are \emph{simple} (have no repeated blocks) can be constructed for sufficiently large $v$ from various PBD$(v,K)$ 
of dimension three, where $K$ is chosen carefully to avoid the need for repeated blocks.  For instance, simple $(v,3,2)$-designs of dimension three exist for all sufficiently large $v \equiv 0,1\pmod{3}$ by taking $K=\{4,6,7\}$, where the needed PBD of dimension three exists by the main result of \cite{DL2}.  Similarly, we could take $K=\{5,7\}$ for simple $(v,3,3)$-designs.  At this time, however, we have no explicit bound on $v$ for existence of such PBDs of dimension three.

\subsection{Idempotent symmetric latin squares}

A \emph{latin square} of order $n$ is an $n \times n$ array on $n$ symbols such that every row and every column is a permutation of the symbols.  A latin square $L$ is \emph{symmetric} if $L_{ij} = L_{ji}$ for any indices $i$ and $j$.  We may assume the set of symbols (and row/column indices) is $[n]:=\{1,\dots,n\}$.   A latin square is \emph{idempotent} if the entry in diagonal cell $(i,i)$ is $i$ for each $i \in [n]$.  

Idempotent symmetric latin squares can be constructed as `back circulants' for all odd integers $n$.  Examples for orders 3 and 5 are shown below.
$$
\begin{array}{|ccc|}
\hline
1&3&2\\
3&2&1\\
2&1&3\\
\hline
\end{array}
\hspace{3cm}
\begin{array}{|ccccc|}
\hline
1&4&2&5&3 \\
4&2&5&3&1\\
2&5&3&1&4\\
5&3&1&4&2\\
3&1&4&2&5\\
\hline
\end{array}
$$

A (latin) \emph{subsquare} is a sub-array which is itself a latin square. Note that such a sub-array need not be on a contiguous set of rows and columns.  For instance, if $H \le G$ are finite groups, the operation table of $H$ is a latin subsquare of the operation table of $G$.

Idempotent latin squares can be joined using a PBD.  In more detail, suppose we have a PBD$(n,K)$, where $K \subseteq \Z_{\ge 3}$.  For every block $B$, let $L^{B}$ be an idempotent latin square on the symbols of $B$. Then we obtain an $n \times n$ idempotent latin square $L$, defined by
\begin{equation}
\label{EQN:ls-glue}
L_{ij} = 
  \begin{cases}
    i, & \text{if $i=j$};\\
    L_{ij}^{B}, & i\neq j, \text{letting $B$ be the block for which $\{i,j\} \subset B$}.
  \end{cases}
\end{equation}

Using this construction in conjunction with Theorem~\ref{pbd345a}, it was shown in \cite{DN} that for $n \ge 48$ there exists an idempotent latin square of order $n$ having the property that any row, column and symbol appear together in a proper latin subsquare. 
We note the following similar consequence of Theorem~\ref{main}(b) for symmetric latin squares having the same property.

\begin{thm}
There exists an idempotent symmetric latin square of order $n$ in which any choice of row, column, and symbol is contained in a proper latin subsquare if $n$ is odd, $n \ge 27$ and $n \not\in \{33,35,37,41,43,47,51\}$.
\end{thm}

\begin{proof}
Given an integer $n$ satisfying the stated conditions, take a PBD$(n,\{3,5\})$ from Theorem~\ref{main}(b).  Build a latin square $L$ of order $n$ as in \eqref{EQN:ls-glue}, where we use the ingredient squares of orders $3$ and $5$ above.  Since these ingredients are symmetric, so is $L$.  The choice of any row, column, and symbol of $L$ amount to a selection of three points of the PBD.  By assumption, these are contained in a proper flat, say $Y$.  Then the restriction of $L$ to rows, columns and symbols indexed by $Y$ is a latin subsquare containing the chosen trio.
\end{proof}


\begin{thebibliography}{99}

\bibitem{BB}
L.M.~Batten and A.~Beutelspacher, The theory of finite linear spaces:
combinatorics of points and lines, \emph{Cambridge University Press}, 1993.

\bibitem{handbook}
C.J.~Colbourn \and J.H.~Dinitz, eds., {\em The CRC Handbook of Combinatorial
Designs}, 2nd edition, CRC Press, Inc., 2006.

\bibitem{D}
A.~Delandtsheer,
Dimensional linear spaces.
{\it Handbook of incidence geometry}, 193--294,
North-Holland, Amsterdam, 1995.


\bibitem{DL2}
P.J.~Dukes \and A.C.H.~Ling, Pairwise balanced designs with prescribed
minimum dimension.
\emph{Discrete Comput. Geom.} 51 (2014), 485--494.

\bibitem{DN}
P.J.~Dukes and J.~Niezen, Pairwise balanced designs of dimension three. \emph{Australas. J. Combin.} 61 (2015), 98--113.

\bibitem{Niezen}
J.~Niezen, Pairwise balanced designs of dimension three. M.Sc. thesis,
University of Victoria, 2013.

\bibitem{T}
L.~Teirlinck, On Steiner spaces.
{\em J. Combin. Theory Ser. A} 26 (1979), 103--114.

\bibitem{RMW2}
R.M.~Wilson, An existence theory for pairwise balanced designs III: Proof of
the existence conjectures.
{\em J. Combin. Theory Ser. A} 18 (1975), 71--79.



\end{thebibliography}
\end{document}